\documentclass[11pt,reqno]{amsart}

\usepackage{amsmath,amssymb,amsthm,epsfig}
\usepackage{hyperref}
\usepackage{color}
\usepackage{ulem}
\usepackage[mathscr]{eucal}
\usepackage[latin1]{inputenc} 
\usepackage{tikz}

\usepackage[nameinlink]{cleveref}
\hypersetup{
	colorlinks,
	linkcolor={blue!50!black},
	citecolor={blue!50!black},
	urlcolor={blue!50!black}
}

\usepackage{esint} 
\usepackage[usenames,dvipsnames]{pstricks}
\usepackage{etoolbox} 

\newtheorem{theorem}{Theorem}
\newtheorem{definition}{Definition}
\newtheorem{lemma}{Lemma}
\newtheorem{proposition}{Proposition}
\newtheorem{corollary}{Corollary}
\newtheorem{remark}{Remark}

\numberwithin{equation}{section}
\numberwithin{theorem}{section}
\numberwithin{definition}{section}
\numberwithin{lemma}{section}
\numberwithin{proposition}{section}
\numberwithin{corollary}{section}
\numberwithin{remark}{section} 
 
\def \Div {\mathrm{div}}
\def \N {\mathbb{N}}
\def \R {\mathbb{R}}
\def \supp {\mathrm{supp } }
\def \diam {\mathrm{diam}}

\def \dist {\mathrm{dist}}
\def \loc {\mathrm{loc}}

\title[Degenerate quenching problems]{Geometric properties of free boundaries in degenerate quenching problems}

\author[D. J. Ara\'ujo]{Dami\~ao J. Ara\'ujo}
\address{Department of Mathematics, Universidade Federal da Para\'iba, 58059-900, Jo\~ao Pessoa-PB, Brazil}{}
\email{araujo@mat.ufpb.br}

\author[R. Teymurazyan]{Rafayel Teymurazyan}
\address{Applied Mathematics and Computational Sciences (AMCS), Compu\-ter, Electrical and Mathematical Sciences and Engineering Division (CEMSE), King Abdullah University of Science and Technology (KAUST), Thuwal, 23955 -6900, Kingdom of Saudi Arabia}{}
\email{rafayel.teymurazyan@kaust.edu.sa} 

\author[J.M.~Urbano]{Jos\'{e} Miguel Urbano}
\address{Applied Mathematics and Computational Sciences (AMCS), Compu\-ter, Electrical and Mathematical Sciences and Engineering Division (CEMSE), King Abdullah University of Science and Technology (KAUST), Thuwal, 23955 -6900, Kingdom of Saudi Arabia and CMUC, Department of Mathematics, University of Coimbra, 3000-143 Coimbra, Portugal}{} 
\email{miguel.urbano@kaust.edu.sa}

\begin{document}

\subjclass[2020]{Primary 35R35. Secondary 35A21, 35J70}




\keywords{Quenching problem; Free boundary; Hausdorff estimates; Degenerate PDEs}

\begin{abstract}

We study minimizers of non-differentiable functionals modeled on the degenerate quenching problem. Our main result establishes the finiteness of the $(n-1)-$dimensional Hausdorff measure of the free boundary. The proof is based on optimal gradient decay estimates obtained from an intrinsic Harnack-type inequality, along with a detailed analysis in a flatness regime, where minimizers enjoy improved regularity. Our arguments provide an alternative proof of classical results of Phillips and, although developed in the degenerate setting, also offer insights relevant to the singular case.

\end{abstract}

\date{\today}

\maketitle

\section{Introduction}\label{s1}

We investigate in this paper minimizers of $p-$energy functionals of the form
$$
v \longmapsto \int_{\Omega} \frac{|D v|^p}{p}+F(x,v) \,dx,
$$
where the potential $F(x,s) \geq 0$ is non-differentiable. 

This class of problems is well-understood in the case $p=2$, particularly in the context of the obstacle problem, for which $F(x,s)=s_+$ (see \cite{C80, CK80} and the book \cite{PSU12}), or the cavity or Bernoulli free boundary problem, featuring the emblematic discontinuous potential $F(x,s)=\chi_{\{s>0\}}$ and investigated by Alt and Caffarelli in \cite{AC81}. An intermediate scenario emerges by interpolating between the two previous cases and considering potentials exhibiting $\gamma-$growth, namely with
\begin{equation} \label{quench}
F(x,s)= s_+^\gamma, \qquad \gamma \in (0,1).
\end{equation}
Often known as Alt-Phillips potentials (\textit{cf.} \cite{AP86, P1, P2}), they give rise to the so-called quenching problem, which models phenomena characterized by abrupt changes in certain quantities across unknown interfaces. 

In all these examples, the lack of differentiability of $F$ significantly reduces the efficiency of the regularization mechanisms for minimizers relative to the classical scenarios. Still in the case $p=2$, it has been shown in \cite{P1} that the optimal regularity class for minimizers is $C_{\rm loc}^{1,\alpha}$, with 
$$
\alpha=\frac{\gamma}{2-\gamma}, \qquad \gamma \in (0,1).    
$$
For the obstacle and cavity problems, corresponding to $\gamma=1$ and $\gamma=0$, minimizers are, respectively, locally of class $C^{1,1}$ and Lipschitz. Concerning the free boundary, in the seminal works \cite{C81} for the obstacle problem, and \cite{P2} for the case $\gamma \in (0,1)$, analytical and geometric measure theory methods were applied to establish that interfaces have finite $(n-1)-$dimensional Hausdorff measure. For the two-phase problem $(p=2)$, the free boundaries are known to be $C^1$ in dimension two (see \cite{LP08}).

For the nonlinear case $p\neq 2$, much less is known, as the non-quadratic growth in the functional complicates the understanding of the regularity and of the geometric properties of minimizers and the free boundary. Obstacle problems, corresponding to $\gamma=1$, were treated in \cite{ALS15, FKR17} for non-zero obstacles $\varphi$. In particular, in \cite{ALS15}, the authors show that a minimizer $u$ is of class $C^{1,p'-1}$ at the free boundary $\partial\{u>\varphi\}$, for $p^{\prime}=\frac{p}{p-1}$. Regularity estimates for the intermediate scenario $\gamma \in (0,1)$ were established in \cite{LQT15} for the sign-changing case. In \cite{ATV22}, potentials with varying degrees of non-differentiability
$$
F(x,s) \sim (s-\varphi(x))_+^\gamma,
$$
were considered, and improved regularity estimates at contact points in $\partial\{u>\varphi\}$ were obtained. For the particular case of an obstacle $\varphi\equiv 0$, it was shown therein that the minimizers are of class $C^{1,\alpha-1}$ at the quenching interface $\partial\{u>0\}$, for
\begin{equation}\label{definition of alpha}
    \alpha=\frac{p}{p-\gamma},
\end{equation}
revealing the precise interplay between the singularity parameter $\gamma$ and the exponent $p$ in determining the regularity of minimizers. Related problems were studied in \cite{LO08, DP05}, where the Lipschitz regularity of minimizers was established using a singular perturbation technique. 

These efforts notwithstanding, the study of the free boundary for problems involving the $p-$Laplace operator remains virtually virgin ground, the only significant contributions being the results in \cite{KKPS00, LS03, FKR17, CLR12, CLRT14} for the $p-$obstacle problem, and in \cite{DPS03} for the cavity problem, where it was shown that near flat points, the free boundary is of class $C^{1,\beta}$, for a certain $\beta\in(0,1)$. 

In this paper, we advance the theory one step forward, extending the analysis to the degenerate case $p>2$ in the quenching scenario \eqref{quench}, for which the corresponding Euler-Lagrange equation is singular since the right-hand side blows up at the free boundary
$\partial\{u>0\}$. We consider non-negative minimizers and chiefly establish the finiteness of the $(n-1)-$dimensional Hausdorff measure of the free boundary, a result that plays a significant role in the context of universally characterizing the ``size" or ``dimension" of the interface. These estimates are closely related to the (lack of) regularity and the geometric properties of the free boundary. 

Our analysis is based on a pointwise gradient control (\Cref{pointwisegradientestimate}) derived from an intrinsic Harnack-type inequality and a suitable application of regularity estimates at points relatively close to the free boundary, recently established in \cite{ATV22}. This gradient control played a key role in \cite{P2} for the case $p=2$, but the proof there crucially relied on the linearity of the operator and does not extend to the degenerate setting. After proving a sharp non-degeneracy estimate (\Cref{sharp non-degeneracy}) and establishing the porosity of the free boundary, we obtain uniform estimates in a flatness regime, where minimizers exhibit improved geometric and analytical properties. These estimates are essential for obtaining refined Hausdorff measure bounds (\Cref{Hausdorff}). We go beyond the linear setting by leveraging the non-variational approach in \cite{AT13} for uniformly elliptic operators to derive these fine regularity estimates in a nonlinear regime. Our strategy thus differs significantly from that of \cite{P2} and requires entirely new ideas. While our results are formulated for the degenerate case $p>2$ due to the dependence on the regularity theory from \cite{ATV22}, they would also hold for the singular case $1<p<2$ were that regularity available. Additionally, they provide an alternative proof of Phillips' classical results, promising broader applicability in the analysis of free boundary problems. 

Recently, the case of a varying exponent, corresponding to $F(x,s) \sim s^{\gamma(x)}$, was investigated in \cite{ASTU26, STU25}. Allowing the singular exponent to vary makes the problem substantially richer, the central difficulty being to quantify how local geometric fluctuations of $\partial\{u>0\}$ influence both the regularity of the minimizer and the structure of the free boundary itself. Regarding the former, optimal regularity of minimizers was proved under a mild continuity assumption on the exponent in \cite{ASTU26} for the linear case $p=2$, and subsequently extended to the $p-$Laplacian in \cite{STU25}. Concerning the latter, a detailed analysis was developed in \cite{ASTU26} in the linear setting, and the free boundary was shown to be locally a $C^{1,\delta}-$surface, except on a negligible singular set of Hausdorff codimension at least $3$. The ideas introduced in this paper may serve as a stepping stone toward extending those results to the degenerate $p-$Laplacian regime.

The paper is organized as follows. In \Cref{prelim}, we introduce notation, rigorously formulate the problem, and present several preliminary results. \Cref{gradsec} provides an optimal pointwise gradient decay based on a refined growth estimate. In \Cref{nondegsec}, we establish a non-degeneracy estimate and, as a consequence, the porosity of the free boundary; we also disclose how minimizers can, in the positivity set, be controlled from below by their distance to the free boundary. The core of the paper lies in the remaining two sections: in \Cref{s5}, we analyze the behavior of minimizers within a universal flatness regime; finally, in \Cref{Hauss sec}, we prove the finiteness of the $(n-1)-$dimensional Hausdorff measure of the free boundary.

\section{Problem setting and preliminary results}\label{prelim}

Let $\Omega$ be an $n-$dimensional bounded domain and $p\ge2$. This restriction on $p$ is used solely to access the regularity results from \cite{ATV22}, which are valid only in this range. However, we emphasize that all other results in this paper also remain valid for $1<p<2$.

For a given non-negative boundary data $0 \leq g\in W^{1,p}(\Omega)$, we consider minimizers of the functional
\begin{equation}\label{1.1}
J(v):=\int_{\Omega} \left( \frac{|D v (x)|^p}{p} + \left[ v_+(x)\right]^\gamma \right) dx,
\end{equation}
among competitors in the set
\begin{equation}\label{minimizationset}
	\mathbb{K}:=\left\{v\in W^{1,p}(\Omega):\,\,v-g\in W^{1,p}_0(\Omega)\right\}.
\end{equation}

The non-differentiability of the functional is the price to be paid to avoid an extra constraint in the set of admissible competitors. In fact, minimizing \eqref{1.1} over \eqref{minimizationset} is equivalent to minimizing 
\begin{equation}\label{altJ}
\tilde{J}(v):=\int_{\Omega} \left( \frac{|D v (x)|^p}{p} + \left[ v(x)\right]^\gamma \right) dx
\end{equation}
over
$$
\tilde{\mathbb{K}}:=\left\{v\in W^{1,p}(\Omega):\,\,v-g\in W^{1,p}_0(\Omega), \ v\geq 0 \mbox{ a.e. in } \Omega\right\},
$$
as proven in \cite[Lemma 1.4]{PSU12}. We observe, in particular, that assuming the boundary datum is non-negative implies the non-negativity of the minimizer.

The corresponding Euler-Lagrange equation is the singular PDE, holding in the non-coincidence set $\{u>0\} \cap \Omega$,
 \begin{equation}\label{EL}
	\Delta_p u :=\Div\left(|D u|^{p-2}D u\right) =\gamma u^{\gamma-1},
\end{equation}
whose right-hand side blows up at free boundary points since $\gamma \in (0,1)$. 

Most of the results in the paper are valid for non-negative \textit{local minimizers} defined as follows. For a ball $B_r \subset\Omega$, set
$$
J_r(v):=\int_{B_r}\left(\frac{|D v(x)|^p}{p}+\left[ v_+(x)\right]^\gamma \right) dx.
$$
\begin{definition}\label{localminimizer}
A nonnegative function $u\in W^{1,p}(\Omega)$ is a local minimizer with respect to a ball $B_r$ if
$$
J_r(u)\le J_r(v),
$$
for all $v\in W^{1,p}(B_r)$ such that $v-u\in W_0^{1,p}(B_r)$. 
\end{definition}

\begin{remark}\label{scaling}
An important property of minimizers is the invariance under scaling. Observe that if $u$ is a local minimizer with respect to a ball $B_r$, then 
$$
u_s(x):= \frac{u(sx)}{s^{\frac{p}{p-\gamma}}}     
$$ 
is a local minimizer with respect to the ball $B_{\frac{r}{s}}$. More precisely,
$$
J_r(u)=s^{\frac{p\gamma}{p-\gamma}}J_{\frac{r}{s}}(u_s).
$$
\end{remark}

For $0 \leq g\in W^{1,p}(\Omega)$, the existence of a minimizer for \eqref{1.1} is established in \cite[Theorem 2.1]{ATV22} (see also \cite{LQT15} for the specific case of a zero obstacle), together with the bound
$$
        \|u\|_\infty\le\|g\|_\infty.
$$

The local $C^{1,\beta}-$regularity for minimizers of the functional
\begin{equation}\label{kappafunctional}
J_{\kappa}(u)=\int_{\Omega}\left(\frac{|D v(x)|^p}{p}+\kappa\left[ v_+(x)\right]^\gamma \right) dx
\end{equation}
is the object of \cite[Theorem 3.1]{ATV22}. More precisely, for every $\Omega^\prime \subset \Omega$, there exists a constant $C>0$, depending on $\dist(\Omega',\partial\Omega)$, $\|u\|_{L^\infty(\Omega)}$, $n$, $p$ and $\gamma$, but independent of $\kappa$, such that \begin{equation}\label{localregularity}
\|u\|_{C^{1,\beta}(\Omega')}\le C,
\end{equation}    
for the sharp exponent
\begin{equation}\label{beta}
\beta:=\min\left\{\alpha_p^-,\frac{\gamma}{p-\gamma}\right\}.
\end{equation}
Here, $\alpha_p$ is the optimal (unknown for $n\geq 3$) H\"older regularity exponent for the gradient of $p-$harmonic functions, and the notation $c^-$, for $c \in \R$, means any real number less than $c$. It is worth noting that, in the case $n=2$, from the results in \cite{IwaMan} (see also \cite{ATU17}), one has $\beta=\frac{\gamma}{p-\gamma}$. Observe also that when $\gamma \searrow 0$, the above regularity decreases to match the Lipschitz regularity obtained in \cite[Theorem 3.3]{DP05}. 

Moreover, in \cite[Theorem 4.1 and Theorem 6.1]{ATV22}, the growth of the minimizer away from free boundary points is revealed to be
\begin{equation}\label{growthestimate}
    u\le C\, r^{\frac{p}{p-\gamma}}, \quad \textrm{in} \ B_r(y),
\end{equation}
for any $r\in(0,r_0)$ and any $y\in\partial\{u>0\}\cap\Omega'$, where $\Omega'\Subset\Omega$, for universal constants $C>0$ and $r_0>0$. The approach is based on geometric tangential analysis and a fine perturbation, combined with an adjusted scaling argument, ensuring that, in the limit, one obtains a linear elliptic equation without the zero-order term. The intuition behind the proof is that the problem behaves essentially as an obstacle problem for a uniformly elliptic operator.

We conclude this section by showing that any minimizer $u$ of \eqref{1.1} is $p-$subharmonic and solves the corresponding Euler-Lagrange equation. As $u\in W^{1,p}(\Omega)$, $p-$subharmonicity is equivalent to proving that $u$ stays below any $p-$harmonic replacement, according to \cite[Chapter 5]{L19}. The proof is known and can be obtained by combining the tools used in \cite{AC81, DP05, LQT15, P1}. We include it here for the reader's convenience in the case $p>2$.

\begin{lemma}\label{subharmonic}
Let $B\subset\Omega$ be a ball, and let $u$ be a minimizer of \eqref{1.1}--\eqref{minimizationset}. Consider the $p-$harmonic replacement of $u$, \textit{i.e.}, the $p-$harmonic function $v$ in $B$ that agrees with $u$ on $\partial B$. Then $u\le v$ in $B$.
\end{lemma}
 
\begin{proof}
Set $w:=\min(u,v)$. We aim to show that, in fact, $w=u$. As $w-u\in W^{1,p}(B)$, one has $D(w-u)=0$ a.e. on $\{w=u\}$, \cite[Lemma A.4]{KS80}. Hence,
\begin{equation}\label{4.2}
\int_{B}|D w|^{p-2}D w\cdot D(w-u)=\int_{B}|D v|^{p-2}D v\cdot D(w-u)=0,
\end{equation}
where the last equality follows from the fact that $v$ is $p-$harmonic and agrees with $u$ on $\partial B$. On the other hand, defining for $0\le s\le 1$, 
$$
u_s(x):=su(x)+(1-s)w(x),
$$ 
recalling \eqref{4.2} and noting that $u_s-w=s(u-w)$, we have
\begin{equation*}
\begin{split}
    & \int_{B}\left(|D u|^p-|D w|^p\right)\\
    &=\int_0^1\frac{d}{ds}\left(\int_B|D u_s|^p\right)\,ds\\
    &= p \int_0^1 ds \int_B |D u_s|^{p-2} D u_s \cdot D(u-w)\\
    &= p \int_0^1 ds \int_B \left(|D u_s|^{p-2}D u_s-|D w|^{p-2} Dw \right) \cdot D(u-w)\\
    &= p \int_0^1 \frac{ds}{s} \int_B \left(|D u_s|^{p-2}D u_s-|D w|^{p-2}Dw \right) \cdot D(u_s-w).
\end{split}
\end{equation*}
Combining this with the well-known inequality
\begin{equation*}
    \left( |\xi|^{p-2}\xi-|\eta|^{p-2}\eta \right) \cdot (\xi-\eta) \geq c \, |\xi-\eta|^p,     
\end{equation*}
where $c>0$ is a constant depending only on $n$ and $p$, we reach
\begin{eqnarray*}
    \int_{B}\left(|D u|^p-|D w|^p\right)&\geq & cp\int_0^1\frac{ds}{s}\int_B|D (u_s-w)|^p\\
    &=& cp\int_0^1s^{p-1}\,ds\int_B|D(u-w)|^p\\
    &=& c \int_B|D(u-w)|^p \geq 0.
\end{eqnarray*}
Using this and the fact that $w\le u$, we obtain
$$
J(w)-J(u)\le\int_B(w_+^\gamma-u_+^\gamma)\le0.
$$
By the minimality of $u$, we have $J(w)=J(u)$, which only holds if $w=u$.
\end{proof}

\begin{proposition}\label{EL1}
If $u$ is a minimizer of \eqref{1.1}--\eqref{minimizationset}, then
\begin{equation}\label{maineq}
\Delta_pu =\gamma u^{\gamma-1}\quad \textrm{ in }\,\{u>0\},
\end{equation}
in the sense of distributions. 
\end{proposition}

\begin{proof}
By \eqref{localregularity}, we have that $u \in C^{1,\beta}_{\mathrm{loc}}(\Omega)$. 
Let $y \in \Omega$ be such that $u(y)>0$. By continuity, there exists $r>0$ such that 
$$u(x) \geq \frac{u(y)}{2} >0, \quad x \in B_r(y).$$
In particular, the term \(u^{\gamma-1}\) remains bounded in \(B_r(y)\).

Let $\xi \in C_0^\infty(B_r(y))$. Then, there exists $\varepsilon_0>0$ such that, for all
$|\varepsilon| \le \varepsilon_0$,
$$
u(x)+\varepsilon\xi(x) \ge \frac{u(y)}{4} > 0, \quad x \in B_r(y),
$$
and hence $u+\varepsilon\xi \in \tilde{\mathbb{K}}$. Since $u$ is a minimizer of \eqref{altJ} over $\tilde{\mathbb{K}}$, it follows that $\varepsilon=0$ is a minimizer of the function
$$
\varepsilon \longmapsto \tilde{J}(u+\varepsilon\xi).
$$
Therefore,
$$
\left.\frac{d}{d\varepsilon} \tilde{J}(u+\varepsilon\xi)\right|_{\varepsilon=0} = 0,
$$
which yields
$$
\int_{B_r(y)}
\left(
|D u|^{p-2} D u \cdot D\xi
+
\gamma u^{\gamma-1} \xi
\right)\,dx = 0.
$$
Since $y \in \{u>0\}$ was arbitrary, the conclusion follows.
\end{proof}
 
\section{A pointwise gradient estimate}\label{gradsec}

In this section, we prove an optimal pointwise gradient decay, which plays an essential role in the analysis of the free boundary. We elaborate on ideas from \cite{AS23} to derive an optimal oscillation estimate. 

\begin{lemma}\label{Harnack type inequality}
If $u$ is a local minimizer of \eqref{1.1} in $B_1$, then there exists $C>0$, depending only on $n$, $p$, $\gamma$ and $\|u\|_{L^\infty(B_1)}$, such that
$$
\sup_{B_r(x)}u\le C\left(r^{\frac{p}{p-\gamma}}+u(x)\right),
$$
for any $x\in B_{1/2}$ and $r\in(0,1/2)$.
\end{lemma}

\begin{proof}
Indeed, if the conclusion fails, then, for each $k \in \N$, there exist points $x_k\in B_{1/2}$, local minimizers $u_k$ in $B_1$, and $r_k\to0$ such that
\begin{equation}\label{contradiction}
    s_k:=\sup_{B_{r_k}(x_k)}u_k > k\left(r_k^{\frac{p}{p-\gamma}}+u_k(x_k)\right).
\end{equation}
Set
$$    
v_k(x):=s_k^{-1}u_k(x_k+r_kx),\,\,\,x\in B_1.
$$
Observe that
$$
\sup_{B_1}v_k=1 \qquad \textrm{and} \qquad v_k(0)+\frac{r_k^{\frac{p}{p-\gamma}}}{s_k} < \frac{1}{k}.
$$
On the other hand, $v_k$ minimizes
$$
J_{k}(v):=\int_{B_1}\left(\frac{|D v (x)|^p}{p}+\frac{r_k^p}{s_k^{p-\gamma}}\,\left[ v (x)\right]^\gamma\right)\,dx,
$$
over $\tilde{\mathbb{K}}$ in $B_1$, and, recalling \eqref{contradiction}, one has
\begin{equation}\label{5.3}
\frac{r_k^p}{s_k^{p-\gamma}}<\frac{1}{k^{p-\gamma}}<1.
\end{equation}
Hence, from \eqref{localregularity}, there exists a function $v_\infty$ defined in $B_1$ such that, up to a subsequence, $v_k \to v_\infty$ and $Dv_k \to Dv_\infty$ locally uniformly in $B_1$. Obviously, $v_\infty(0)=0$ and
\begin{equation}\label{rain}
    \sup_{B_1} v_\infty = 1.
\end{equation}

Furthermore, since
$$
J_{k}(v_k) \le J_{k}(v_k + \varepsilon \xi),
\quad \forall \varepsilon > 0,\ \forall \xi \in C_0^\infty(B_1),\ \xi \ge 0,
$$
then, in view of \eqref{5.3}, one has
$$
J_0(v_\infty) \le J_0(v_\infty + \varepsilon \xi),
$$
where
$$
J_0(v) := \int_{B_1} \frac{|D v|^p}{p}\, dx.
$$
This yields
$$
\int_{B_1} |D v_\infty|^{p-2} D u_\infty \cdot D \xi \, dx \ge 0,
\quad \forall \xi \in C_0^\infty(B_1),\ \xi \ge 0,
$$
and so, $v_\infty$ is $p-$superharmonic in $B_1$. Hence, since $v_\infty(0)=0$, we conclude from the strong minimum principle that $v_\infty \equiv0$, which contradicts \eqref{rain}.
\end{proof}

The following result, an immediate consequence of the previous lemma, is an intrinsic Harnack-type inequality.

\begin{corollary}\label{Harnack}
Let $u$ be a local minimizer of \eqref{1.1} in $B_1$, with $\|u\|_{L^\infty(B_1)}\le 2^{\frac{\gamma-p}{p}}$. There exists a constant $C>0$, depending only on $n$, $p$ and $\gamma$, such that, for each $x \in B_{1/2}$,
$$
\sup_{B_r(x)}u\le Cu(x),
$$
for $r:=\left[u(x)\right]^{\frac{p-\gamma}{p}}< \frac{1}{2}$. 
\end{corollary}

As a consequence, we obtain the following pointwise gradient estimate.

\begin{theorem}\label{pointwisegradientestimate}
If $u$ is a local minimizer of \eqref{1.1} in $B_1$, with $\|u\|_{L^\infty(B_1)}\le 2^{\frac{\gamma-p}{p}}$, then there exists a constant $C>0$, depending only on $n$, $p$, and $\gamma$, such that
$$
|D u(x)|^p\leq C\|u\|_{L^\infty(B_1)}^{p-\gamma}\, u^{\gamma}(x),\quad x\in B_{1/2}.
$$
\end{theorem}

\begin{proof} 
Since, locally, we have $\partial\{u>0\}\subset\{|Du|=0\}$, it is enough to prove the result for $x\in B_{1/2}$ such that $u(x)>0$. 

If $\|u\|_{L^\infty(B_1)} = 2^{\frac{\gamma-p}{p}}$, then setting 
$$
    v(y):=r^{-\frac{p}{p-\gamma}} \, u(x+ry), \quad y \in B_1,
$$
where  $r:=\left[u(x)\right]^{\frac{p-\gamma}{p}}<\frac{1}{2}$, and applying \Cref{Harnack}, we get 
$$
\sup_{B_1}v\le C,
$$ 
for some universal constant $C>0$. On the other hand, since $v$ is a local minimizer in $B_1$, from \eqref{localregularity}, we get the Lipschitz bound
$$
|D v(0)|\le C'C,
$$
for a universal constant $C'>0$. Observe now that 
$$
|D v(0)|=r^{-\frac{\gamma}{p-\gamma}}|D u(x)|=|D u(x)|u^{-\frac{\gamma}{p}}(x),
$$
and the result follows.

If $\|u\|_{L^\infty(B_1)} < 2^{\frac{\gamma-p}{p}}$, then set 
$$
v(x):= \kappa u(x) \quad \mbox{in} \quad B_1,
$$
where $\kappa:=2^{\frac{\gamma-p}{p}}\|u\|_{L^\infty(B_1)}^{-1}$ and observe that $v$ is a minimizer of
$$
\int_{B_1}\left(\frac{|D v (x)|^p}{p}+\kappa^{\gamma-p}\,[v_+(x)]^\gamma\right)\,dx.
$$
Hence, by \eqref{kappafunctional}-\eqref{localregularity}, $v$ is locally of class $C^{1,\beta}$. Since $\|v\|_{L^\infty(B_1)}=2^{\frac{\gamma-p}{p}}$, arguing as before leads to
$$
\kappa^{\frac{p-\gamma}{p}}|D u(x)|u^{-\frac{\gamma}{p}}(x)=|D v(x)|v^{-\frac{\gamma}{p}}(x)\le C.
$$
\end{proof}

\section{Non-degeneracy and porosity}\label{nondegsec}

In this section, we prove non-degeneracy and positive density results for minimizers of \eqref{1.1}, thereby establishing the porosity of the free boundary. 
\begin{proposition}\label{non-degeneracy}
If $u$ is a minimizer of \eqref{1.1} and $x_0\in\overline{\{u>0\}}$, then there exists a constant $c>0$, depending only on $n$, $p$ and $\gamma$, such that    
\begin{equation}\label{non-degeneracyestimate}
    \sup_{\partial B_r(x_0)}u\ge cr^{\frac{p}{p-\gamma}},
\end{equation}
for any $r<\dist(x_0,\partial \Omega)$.
\end{proposition}

\begin{proof}
By continuity, it is enough to prove the result for $x_0\in\{u>0\}$. Fix $r<\dist(x_0,\partial \Omega)$ and set
$$\lambda:=\frac{p-\gamma}{p-1} >1.$$
A direct formal calculation shows that, in $\{u>0\}\cap B_r(x_0)$, one has
\begin{equation*}\label{transformation}
\Delta_p(u^\lambda) = \lambda^{p-1}u^{-\gamma} \left[ (1-\gamma)|D u|^p+u \Delta_p u \right],
\end{equation*}
which, recalling \Cref{EL1}, leads to    
$$\Delta_p(u^\lambda)=\lambda^{p-1}u^{-\gamma}\left[(1-\gamma)|D u|^p+\gamma u^\gamma\right]\ge\gamma\lambda^{p-1}>0.$$   
Moreover,
$$\Delta_p(c|x-x_0|^{\frac{p}{p-1}})=c^{p-1}n\left(\frac{p}{p-1}\right)^{p-1}.$$
Choosing $c>0$ universally small, we ensure
$$
\Delta_p(u^\lambda)\ge\Delta_p(c|x-x_0|^{\frac{p}{p-1}})\,\,\,\textrm{ in }\,\,\,\{u>0\}\cap B_r(x_0).
$$
Furthermore, $u^\lambda(x_0)>0=c|x_0-x_0|^{\frac{p}{p-1}}$. By the comparison principle, there is $y\in\partial\left(\{u>0\}\cap B_r(x_0)\right)$ such that $$u^\lambda(y)\ge c|y-x_0|^{\frac{p}{p-1}}.$$
Since at $\partial\{u>0\}\cap B_r(x_0)$, one has
$$u^\lambda(x)=0<c|x-x_0|^{\frac{p}{p-1}},$$ 
necessarily $y\in\{u>0\}\cap\partial B_r(x_0)$. Hence, 
$$
u^\lambda(y)\ge c|y-x_0|^{\frac{p}{p-1}}=cr^{\frac{p}{p-1}},
$$
and the result follows.
\end{proof}

\begin{remark}\label{REM}
We highlight that the pointwise computations for $\Delta_p(u^\lambda)$ in the above proof are standard (\textit{cf.} \cite{LS03}) and valid, in the weak sense, throughout the entire positivity set $\{u>0\}$. Indeed, let $\phi\in C_0^\infty(\{u>0\})$ be a test function and define
$$
\eta:=u^{1-\gamma}\phi.
$$
Since $\supp(\phi)\Subset\{u>0\}$, there exists a subdomain $\Omega'\Subset\{u>0\}$ such that $\supp(\phi)\subset\Omega'$. Hence, for some constant $c>0$, we have
$u\ge c$ in $\Omega'$. In particular, $u^{1-\gamma} \in L^\infty(\Omega')$. Moreover,
$u\in W^{1,\infty}(\Omega')$, so
$$
\eta\in W^{1,p}_0(\Omega').
$$
Testing \eqref{maineq} against $\eta$, gives
$$
\int_{\Omega'} |Du|^{p-2}Du\cdot D(u^{1-\gamma} \phi) \, dx = -\int_{\Omega'} \gamma u^{\gamma-1}u^{1-\gamma} \phi\,dx = - \gamma \int_{\Omega'} \phi\,dx,
$$
and, substituting the expression for $D\eta$, we obtain
$$
\int_{\Omega'} u^{1-\gamma} |Du|^{p-2}Du\cdot D\phi \, dx + (1-\gamma) \int_{\Omega'} u^{-\gamma}\phi|Du|^p\,dx = - \gamma \int_{\Omega'} \phi\,dx.
$$
On the other hand,
$$
|D(u^\lambda)|^{p-2}D(u^\lambda) = \lambda^{p-1} u^{1-\gamma} |Du|^{p-2}Du,
$$
and, combining the previous identities, we arrive at
\begin{eqnarray*}
& \displaystyle \int_{\{u>0\}} |D(u^\lambda)|^{p-2}D(u^\lambda)\cdot D\phi\,dx \nonumber \\
& = \lambda^{p-1} \displaystyle\int_{\{u>0\}} \left( -\gamma 
-(1-\gamma) u^{-\gamma}|Du|^p \right) \phi\,dx,
\end{eqnarray*}
for every $\phi\in C_0^\infty(\{u>0\})$,
which is precisely the desired weak formulation for $\Delta_p(u^\lambda)$ in $\{u>0\}$. Notably, the argument only requires the $C^1-$regularity of $u$, which is guaranteed by the local $C^{1,\beta}-$estimates for minimizers.
\end{remark}

\begin{remark}
Note that from \eqref{growthestimate}, we already know that in $B_r(x_0)$, with $x_0\in\partial\{u>0\}$, the supremum of a minimizer grows at most like $r^{\frac{p}{p-\gamma}}$. \Cref{non-degeneracy} reveals that it grows exactly as $r^{\frac{p}{p-\gamma}}$.
\end{remark}

\begin{corollary}\label{integral non-degeneracy}
If $u$ is a minimizer of \eqref{1.1}, and $x_0\in\overline{\{u>0\}}$, then there exists a constant $c^\ast >0$, depending only on $n$, $p$, $\gamma$ and $\|u\|_{L^\infty(B_1)}$, such that
$$
\fint_{B_r(x_0)}u\,dx\ge c^\ast\,r^{\frac{p}{p-\gamma}},
$$
for any $r<\dist(x_0,\partial\Omega)$.
\end{corollary}

\begin{proof}
    Indeed, given $r<\dist(x_0,\partial\Omega)$, \Cref{non-degeneracy} guarantees the existence of $y\in \partial B_{r}(x_0)$ such that 
\begin{equation}\label{diplomat}
u(y)\ge c\,r^{\frac{p}{p-\gamma}}>0.
\end{equation}
Recall \eqref{definition of alpha} and fix 
$$
\rho = \frac12 \left( \frac{c}{C} \right)^{\frac{1}{\alpha}},
$$
where $c$ is the constant from \Cref{non-degeneracy} and $C$ is the constant from \Cref{Harnack type inequality}. For any $z \in B_{\rho r}(y)$, we have, using \eqref{diplomat} and \Cref{Harnack type inequality},
$$
c\,r^{\alpha} \leq u(y) \leq \sup_{B_{\rho r}(z)}u  \le  C \left[ (\rho r)^{\alpha} + u(z) \right].
$$
Hence, 
$$
u(z) \geq \frac{1}{C}\left( cr^{\alpha}-C (\rho r)^{\alpha} \right) =  \frac{c}{C} \left( 1- \left( \frac12 \right)^{\alpha} \right) r^{\alpha} = \tilde{c}\, r^{\alpha},
$$
for any $z \in B_{\rho r}(y)$, with $\tilde{c}$ depending only on $n$, $p$, $\gamma$ and $\|u\|_{L^\infty(B_1)}$.
Thus,
$$
\fint_{B_r(x_0)}u\, dx \geq \left( \frac{\rho}{2} \right)^{n} \fint_{B_{\rho r}(y)\cap B_r(x_0)}u \, dx \geq \left( \frac{\rho}{2} \right)^{n} \tilde{c} \, r^{\alpha},
$$
and the result follows with $c^\ast := \left( \frac14 \left( \frac{c}{C} \right)^{1-\frac{\gamma}{p}} \right)^{n} \tilde{c}$.
\end{proof}

\begin{corollary}\label{positivedensity}
If $u$ is a minimizer of \eqref{1.1}, then, for every $\Omega'\Subset\Omega$, there exists a constant $\tau \in(0,1)$, depending only on $\Omega'$, $n$, $p$ and $\gamma$, such that, for any $x_0\in\Omega'\cap\partial\{u>0\}$ and any $r>0$ small enough, one has
$$
\frac{|\{u>0\}\cap B_r(x_0)|}{|B_r(x_0)|}\ge\tau.
$$
\end{corollary}

\begin{proof}
Indeed, as in the proof of \Cref{integral non-degeneracy}, $u>0$ in $B_{\rho r}(y)$, where $y\in\partial B_r(x_0)$. Therefore, 
$$
\frac{|\{u>0\}\cap B_r(x_0)|}{|B_r(x_0)|}\ge\frac{|B_{\rho r}(y)\cap B_r(x_0)|}{|B_r(x_0)|} \geq \left( \frac{\rho}{2} \right)^{n} =:\tau.
$$    
\end{proof}

We recall the definition of a porous set to state another consequence of the non-degeneracy estimate.

\begin{definition}
A set $E\subset\R^n$ is called porous, with porosity constant $\delta>0$, if there exists a constant $\rho>0$ such that, for each $x\in E$ and $r\in(0,\rho)$, there exists a $y\in\R^n$ such that 
$$
B_{\delta r}(y)\subset B_r(x)\setminus E.
$$
\end{definition}

The Hausdorff dimension of a porous set does not exceed $n-C\delta^n$, where $C>0$ is a constant depending only on $n$ (see, for example, \cite{MV87}). Hence, the Lebesgue measure of a porous set is zero.

\begin{corollary}\label{porosity}
Let $u$ be a minimizer of \eqref{1.1}. If $x_0\in\Omega$ and $r>0$ are such that $B_{2r}(x_0)\subset\Omega$, then the set 
$$
E:=\partial\{u>0\}\cap\overline{B_r(x_0)}
$$ 
is porous. Hence, the free boundary $\partial\{u>0\}$ is a set of Lebesgue measure zero.
\end{corollary}

\begin{proof}
Let $x\in E$. We have $B_{r/2}(x)\subset B_{2r}(x_0)\subset\Omega$. From \Cref{non-degeneracy}, there exists $y\in\partial B_{r/2}(x)$ such that 
$$
u(y)\ge cr^{\frac{p}{p-\gamma}},
$$
for a constant $c>0$ depending only on $n$, $p$ and $\gamma$. Hence, 
$$
y\in B_{2r}(x_0)\cap\{u>0\}.
$$
Set $d(y):=\dist \left(  y, \overline{B_{2r}(x_0)}\setminus \{u>0\} \right)$, then \eqref{growthestimate} provides
$$
u(y)\le C\left[d(y)\right]^{\frac{p}{p-\gamma}},
$$
for a constant $C>0$ depending only on $n$, $p$ and $\gamma$. Therefore, setting 
$$
\delta:=\min\left\{\frac{1}{2},\left[cC^{-1}\right]^{\frac{p-\gamma}{p}}\right\}<1,
$$ 
we have
$$
d(y)\ge\delta r.
$$
Hence, $B_{\delta r}(y)\subset B_{d(y)}(y)\subset\{u>0\}$. In particular,  
$$
B_{\delta r}(y)\cap B_r(x)\subset\{u>0\}.
$$
On the other hand, if $z\in[x,y]$ is such that $|z-y|=\delta r/2$, then
$$
B_{(\delta/2)r}(z)\subset B_{\delta r}(y)\cap B_r(x).
$$
Indeed, if $z^*\in B_{(\delta/2)r}(z)$, then
$$
|z^*-y|\le|z^*-z|+|z-y|<\frac{\delta r}{2}+\frac{\delta r}{2}=\delta r
$$
and, since $|x-y|=|x-z|+|z-y|$,
$$
|z^*-x|\le|z^*-z|+|x-y|-|z-y|<\frac{\delta r}{2}+r-\frac{\delta r}{2}=r.
$$
Thus,
$$
B_{(\delta/2)r}(z)\subset B_{\delta r}(y)\cap B_r(x)\subset B_r(x)\setminus\partial\{u>0\}\subset B_r(x)\setminus E,
$$
\textit{i.e.}, $E$ is porous with porosity constant $\delta/2$. 
\end{proof}

We close this section with yet another non-degeneracy result revealing that minimizers can be controlled from below by their distance to the free boundary in the positivity set.
\begin{remark}\label{important remark}
    Since minimizers of    
    \begin{equation}\label{mumin}
    \int_{B_1} \left( \frac{|Du|^p}{p}+\mu u_+^{\gamma} \right)\,dx,
\end{equation}
where $\mu\in(0,1]$ are locally of class $C^{1,\beta}$, for a certain $\beta\in(0,1)$ independent of $\mu$ \cite[Theorem 3.1]{ATV22}, the estimate in \Cref{Harnack type inequality} remains true for such minimizers and is independent of $\mu$. Observe that \Cref{non-degeneracy} also remains valid for minimizers of \eqref{mumin}, independent of $\mu$.
\end{remark}

\begin{theorem}\label{sharp non-degeneracy}
If $u$ is a local minimizer of \eqref{1.1} in $B_1$, then there exists a universal constant $c_0>0$ such that, in $\{u>0\}$, one has
\[
    u(x) \ge c_0 \,\left[d(x)\right]^\frac{p}{p-\gamma},
\]
where 
\begin{equation}\label{definition of d}
    d(x):=\dist \left( x, \partial\{u>0\} \right).
\end{equation}
\end{theorem}

\begin{proof}
In light of \Cref{important remark}, we prove the result for minimizers of \eqref{mumin}. More precisely, we show that if $u$ is a minimizer of \eqref{mumin}, then there exist universal constants
\( \mu_0>0 \) and \( c_0>0 \) such that the conclusion holds for $\mu \leq \mu_0$. Indeed, if not, then there exists a sequence $u_k$ of minimizers of \eqref{mumin}, with $\mu=1/k$ and $x_k\in \{u_k > 0\}$, satisfying 
\[
d_k:=\dist \left( x_k, \partial\{u_k>0\} \right)\longrightarrow  0,
\]
such that 
\begin{equation}\label{pointwise}
    u_k(x_k) \le \frac{1}{k} d_k^\alpha,
\end{equation}
where $\alpha=\frac{p}{p-\gamma}$, as defined by \eqref{definition of alpha}. Set now
\begin{equation*}
    v_k(y):= \frac{1}{d_k^\alpha} u_k(x_k + d_k y).
\end{equation*}
As observed in \Cref{scaling}, $v_k$ is a minimizer of \eqref{mumin} in $B_{1/d_k}$. By \Cref{Harnack type inequality} and \eqref{pointwise}, one has
\[
0 \leq v_k(y) = \frac{1}{d_k^\alpha} u_k(x_k + d_k y) \leq \frac{1}{d_k^\alpha}C(d_k^\alpha + u_k(x_k)) \leq 2C.
\]
Recalling \Cref{non-degeneracy} and \Cref{important remark}, we have
\begin{equation}\label{nondeglim}
    \sup_{B_{1/2}} v_k \geq c_0 \left(\frac{1}{2} \right)^\alpha,
\end{equation}
for a universal constant $c_0 > 0$. Then, up to a subsequence, $v_k \to w$ in $C_{\loc}^{1,\beta}$, as $k \to \infty$ \cite[Theorem 3.1]{ATV22}, with $w$ satisfying
\[
\int_{B_1} |D w|^{p-2} D w \cdot D \varphi \, dx \geq 0, \quad \forall \varphi \in C^\infty_0(B_1), \quad \varphi \geq 0.
\]
Thus, $w\ge0$ is a $p-$superharmonic function in $B_1$. Furthermore, \eqref{pointwise} implies $w(0) = 0$, so, by the strong maximum principle \cite[Corollary 2.22]{L19}, $w\equiv0$. But, from \eqref{nondeglim}, we get 
\[
\sup_{B_{1/2}} w \geq c_0 \left(\frac{1}{2} \right)^\alpha,
\]
reaching a contradiction.

It remains to note that if $u$ is a minimizer of \eqref{1.1}, then $\tilde{u}:=\mu_0^{\frac{1}{p-\gamma}}u$ is a minimizer of \eqref{mumin} for $\mu=\mu_0$.
\end{proof}

\section{Universal flatness regime}\label{s5}

In this section, we study the behavior of minimizers in a flatness regime, where they exhibit improved geometric and analytic properties, which are crucial for deriving precise Hausdorff measure estimates, as detailed in the next section. The analysis is particularly relevant, as regions where minimizers are flat are close to the free boundary.

We first show that, in a universally flat regime, the function \( u^{\frac{p - \gamma}{p}} \) is \( p \)-subharmonic. This property is instrumental in analyzing the behavior of \( u \) near the free boundary. 

\begin{proposition}\label{subsol}
If $u$ is a local minimizer of \eqref{1.1} in $B_1$, then there exists $\delta_\star>0$, depending only on $n$, $p$ and $\gamma$, such that 
$$
\Delta_p\left(u^{\frac{p-\gamma}{p}}\right) \geq 0 \quad \mbox{in} \quad \left\{0<u\leq\delta_\star^{\frac{p}{p-\gamma}}\right\} \cap B_{1/2}.
$$
\end{proposition}

\begin{proof}
Using \Cref{EL1} and arguing as in \Cref{REM}, a direct computation yields, in $\{u>0\}$,
\begin{equation}\label{trieste2024}
\Delta_p(u^{\frac{1}{\alpha}})=\frac{1}{\alpha^{p-1}}u^{(\frac{1}{\alpha}-1)(p-1)-1}\left[\left(\frac{1}{\alpha}-1\right)(p-1)|Du|^p+\gamma u^\gamma\right],
\end{equation}
where $\alpha>1$ is defined by \eqref{definition of alpha}.
Recall that, from \Cref{pointwisegradientestimate},
\begin{equation}\label{magicchoice}
|Du|^p \le C\delta^{p} u^{\gamma} \quad \mbox{in} \quad \left\{ u\leq \delta^\alpha \right\} \cap B_{1/2}, 
\end{equation}
for a constant $C>0$ depending only on $n$, $p$ and $\gamma$. For the choice 
\begin{equation}\label{definition of delta star}
    \delta = \delta_\star:= \left(\frac{\gamma\alpha}{C(\alpha-1)(p-1)}\right)^{\frac{1}{p}} = \left(\frac{p}{C(p-1)}\right)^{\frac{1}{p}},
\end{equation}
estimate \eqref{magicchoice} ensures that the negative contribution arising from the gradient term in \eqref{trieste2024} is dominated by the positive zero-order term. More precisely, in \(\{u \le \delta_\star^\alpha\} \cap B_{1/2}\) we have
$$
\left(\frac{1}{\alpha}-1\right)(p-1)|Du|^p+\gamma u^\gamma = -\gamma \frac{(p-1)}{p} |Du|^p+\gamma u^\gamma \ge -\gamma u^\gamma +\gamma u^\gamma = 0,
$$
which combined with \eqref{trieste2024} concludes the proof.
\end{proof}

Next, we show that solutions restricted to balls centered at the free boundary (and with universally small radii) fall within the smallness regime described in \Cref{subsol}. 

\begin{proposition}\label{balls in flatness}
If $u$ is a local minimizer of \eqref{1.1} in $B_1$, then there exists $r_\star>0$, depending only on $n$, $p$, $\gamma$ and $\|u\|_{L^\infty (B_1)}$, such that, for each $x \in\partial\{u>0\}$ and $r \leq r_\star$, we have
$$
B_r(x)\subset \left\{ u \leq \delta_\star^{\frac{p}{p-\gamma}}\right\},
$$
where $\delta_\star>0$ is defined by \eqref{definition of delta star}.
\end{proposition}

\begin{proof}
We know from \eqref{growthestimate} that there exist universal constants $C>0$ and $r_0>0$ such that
$$   
u\le C\, r^{\frac{p}{p-\gamma}}, \quad \textrm{in} \ B_r(x),
$$
for any $r\in(0,r_0)$ and $x\in \partial\{u>0\}$. Let 
\begin{equation}\label{definition of r star}
    r_\star:= \min \left\{ r_0, \delta_\star C^{-\frac{p-\gamma}{p}} \right\}.
\end{equation}
Now, if $x\in \partial\{u>0\}$, $r\le r_\star$ and $y\in B_r(x)$, we have
$$
u(y) \leq C\,r^{\frac{p}{p-\gamma}} \leq C\,{r_\star}^{\frac{p}{p-\gamma}} \leq C\,\left( \delta_\star C^{-\frac{p-\gamma}{p}}\right)^{\frac{p}{p-\gamma}} = \delta_\star^{\frac{p}{p-\gamma}},
$$
which gives the desired inclusion.
\end{proof}

\begin{figure}[ht!]
\centering
\begin{tikzpicture}[scale=2.0]

\draw (-3,3) -- (-3,0) -- (3,0) -- (3,3) -- (-3,3);

\filldraw [black] (-0.2,0.75) circle (0.0pt) node[below left, gray] {$0<u\leq \delta_\star^{\frac{p}{p-\gamma}}$};

\filldraw [black] (-1.1,2.3) circle (0.0pt) node[below left, black] {$u> \delta_\star^{\frac{p}{p-\gamma}}$};

\filldraw [black] (2.4,1.6) circle (0.0pt) node[below left, gray] {$u=0$};

\filldraw [black] (0.64,1.64) circle (0.0pt) node[below left, gray] {$r_\star$};

\filldraw [black] (0,3.6) circle (0.0pt) node[below left, gray] {$\Delta_p(u^{\frac{p-\gamma}{p}})\geq 0$};

\filldraw [gray, opacity=0.7] (-3,3) -- (-3,0) -- (-2.4,0) to [out=80,in=0] (-2.5,0.7) to [out=45,in=180] (-1,1)  to [out=30,in=-30] (0.2,2.5) -- (-0.1,2.3) -- (0.2,3) -- (-3,3);

\filldraw [gray, opacity=0.4] (-3,3) -- (-3,0) -- (1.5,0) -- (1,0.2) -- (0.4,1) to [out=60,in=180] (1.1,1.5)  to [out=90,in=180] (2.7,3) -- (-3,3);

\filldraw [-, gray, opacity=0.7] (-0.3,3.2) -- (0,3.2);

\filldraw [->, gray, opacity=0.7] (0,3.2) -- (0.9,1.7);

\draw[gray!100, fill=gray!60, opacity=0.3] (0.8,1.43) circle (0.73cm);

\filldraw [-, gray, opacity=0.7, dashed] (0.29,1.9) -- (0.8,1.43);

\filldraw[gray!80] (0.8,1.43) circle (0.5pt) node[below] {$x$};
\end{tikzpicture}
\caption{Behavior of minimizers in the flatness regime.}
\end{figure}

\medskip

As a primary consequence of the propositions above, we deduce an integral estimate for \( u^{-\gamma}|Du|^p \) in the flatness regime. This result plays a crucial role in deriving Hausdorff measure estimates for the free boundary. 

\begin{lemma}\label{integral estimate in strip}
Let $\alpha>1$, $\delta_\star>0$ and $r_\star>0$ be defined, respectively, by \eqref{definition of alpha}, \eqref{definition of delta star} and \eqref{definition of r star}. If $u$ is a minimizer of \eqref{1.1} in $B_1$, then there exists $C>0$, depending only on $n$, $p$, $\gamma$ and $\|u\|_{L^\infty (B_1)}$, such that, for every $0<r\leq r_\star$ and $0<\delta\le\delta_\star$, we have
\begin{equation}\label{estimate on strip}
\displaystyle\int\limits_{\{0<u\leq\delta^\alpha\}\cap B_r(x_0)}u^{-\gamma}|Du|^p\,dx \leq C\delta r^{n-1},
\end{equation}
for $x_0\in\partial\{u>0\} \cap B_{1/2}$.
\end{lemma}

\begin{proof}
Let $\varphi$ be a non-decreasing Lipschitz function in $\mathbb{R}^+$ such that $\varphi(t)\le t$, for any $t\ge0$, and  
\begin{equation*}
\varphi(t)=
\begin{cases}
0 \quad \textrm{if} \quad t\le\frac{1}{2},\\
t \quad \textrm{if} \quad t\ge 1.
\end{cases}
\end{equation*}
Fix $\delta\leq\delta_\star$ and, for each $0<\varepsilon<\delta$, set
\begin{equation*}\label{auxiliary}    \varphi_\varepsilon(t):=\varepsilon \varphi\left(\frac{t}{\varepsilon}\right)
\end{equation*}
and define
\begin{equation}\label{definition of omega}
    \omega:=\min \left\{ u^\frac{1}{\alpha},\delta \right\}.
\end{equation}
Now, fix $r\le r_\star$. By \Cref{balls in flatness}, we have 
\begin{equation} \label{keira}
B_r(x_0)\subset \left\{ u \leq \delta_\star^{\frac{p}{p-\gamma}}\right\}.
\end{equation}
Observe first that
\begin{equation}\label{equality}
    \begin{split}    &\int\limits_{B_r(x_0)}|Du^{\frac{1}{\alpha}}|^{p-2}D(u^\frac{1}{\alpha})\cdot D(\varphi_\varepsilon(\omega))\,dx\\
    &=\int\limits_{\left\{u<\varepsilon^\alpha\right\}\cap B_r(x_0)}|Du^{\frac{1}{\alpha}}|^{p-2}D(u^\frac{1}{\alpha})\cdot D(\varphi_\varepsilon(\omega))\,dx\\
    &\quad+\int\limits_{\left\{\varepsilon^\alpha \leq u \leq \delta^\alpha\right\}\cap B_r(x_0)}|Du^{\frac{1}{\alpha}}|^{p}\,dx.
    \end{split}
\end{equation}
On the other hand, (formal) integration by parts gives
\begin{equation}\label{integration by parts}
\begin{split}
&\int\limits_{B_r(x_0)}|Du^{\frac{1}{\alpha}}|^{p-2}D(u^\frac{1}{\alpha})\cdot D(\varphi_\varepsilon(\omega))\,dx\\
&=\frac{1}{r}\int\limits_{\partial B_r(x_0)}\varphi_\varepsilon(\omega)\,|Du^{\frac{1}{\sigma}}|^{p-2} D(u^\frac{1}{\alpha})\cdot(x-x_0)\,dx-\int\limits_{B_r(x_0)}\varphi_\varepsilon(\omega)\,\Delta_p(u^{\frac{1}{\alpha}})\,dx.
\end{split}
\end{equation}
Since $\varphi_\varepsilon(\omega)\equiv 0$ in $\{u\le(\varepsilon/2)^\alpha\}$, we have
\begin{equation}\label{splitting the integral}
\int\limits_{B_r(x_0)}\varphi_\varepsilon(\omega)\,\Delta_p(u^{\frac{1}{\alpha}})\,dx=\int\limits_{B_r(x_0) \cap  \{ u> \left(\varepsilon/2\right)^\alpha\}}\varphi_\varepsilon(\omega)\,\Delta_p(u^{\frac{1}{\alpha}})\,dx.
\end{equation}
But, recalling \eqref{keira}, 
$$B_r(x_0) \cap  \left\{ u> \left(\varepsilon/2\right)^\alpha \right\} \subset \left\{ \left(\varepsilon/2\right)^\alpha < u \leq \delta_\star^{\frac{p}{p-\gamma}} \right\}$$
and \Cref{subsol} gives
$$
\Delta_p(u^{\frac{1}{\alpha}})\ge 0, \quad \mbox{in } B_r(x_0) \cap  \left\{ u> \left(\varepsilon/2\right)^\alpha \right\}.
$$
Observe also that, from \eqref{trieste2024} and \eqref{magicchoice}, in $B_r(x_0) \cap  \left\{ u> \left(\varepsilon/2\right)^\alpha \right\}$, we have
$$
\Delta_p(u^{\frac{1}{\alpha}})\le C u^{(\frac{1}{\alpha}-1)(p-1)-1+\gamma} = C u^{-\frac{1}{\alpha}} \le\frac{2C}{\varepsilon},
$$
for a constant $C>0$ depending only on $n$, $p$ and $\gamma$. Since
$$
0\le\varphi_\varepsilon(\omega)\le \omega\le\delta,
$$
from \eqref{splitting the integral}, we get
$$
0\le\int\limits_{B_r(x_0)}\varphi_\varepsilon(\omega)\,\Delta_p(u^{\frac{1}{\alpha}})\,dx\le\frac{C\delta}{\varepsilon},
$$
so the last term in \eqref{integration by parts} is finite and non-negative. Therefore,
\begin{equation*}
\begin{split}
&\int\limits_{B_r(x_0)}|Du^{\frac{1}{\alpha}}|^{p-2}D(u^\frac{1}{\alpha})\cdot D(\varphi_\varepsilon(\omega))\,dx\\
&\le\frac{1}{r}\int\limits_{\partial B_r(x_0)}\varphi_\varepsilon(\omega)\,|Du^{\frac{1}{\alpha}}|^{p-2}D(u^\frac{1}{\alpha})\cdot(x-x_0)\,dx.
\end{split}
\end{equation*} 
Thus, combining the last inequality with \eqref{equality}, we obtain 
\begin{equation}\label{integral estimate}
\begin{split}
    & \int\limits_{\left\{\varepsilon^\alpha\le u\le\delta^\alpha\right\}\cap B_r(x_0)}|Du^{\frac{1}{\alpha}}|^{p}\,dx \\
    & \le \delta \int\limits_{\partial B_r(x_0)}|Du^{\frac{1}{\alpha}}|^{p-1}\,dx-\int\limits_{\left\{u<\varepsilon^\alpha\right\}\cap B_r(x_0)}|Du^{\frac{1}{\alpha}}|^{p}\varphi_\varepsilon'(u^{\frac{1}{\alpha}})\,dx \\
    &\le\delta \int\limits_{\partial B_r(x_0)}|Du^{\frac{1}{\alpha}}|^{p-1}\,dx,
\end{split}
\end{equation}
where in the last estimate, we used the fact that $\varphi_\varepsilon'=\varphi'\ge0$. Finally, since
\begin{equation}\label{Gradient of the potential}
    Du^{\frac{1}{\alpha}}=\frac{1}{\alpha}u^{-\frac{\gamma}{p}}Du,
\end{equation}
it follows from \eqref{integral estimate} and \Cref{pointwisegradientestimate} that
\begin{equation*}
    \begin{split}
        \int\limits_{\left\{\varepsilon^\alpha\le u\le\delta^\alpha\right\}\cap B_r(x_0)}u^{-\gamma}|Du|^p\,dx&=\alpha^p\int\limits_{\left\{\varepsilon^\alpha\le u\le\delta^\alpha\right\}\cap B_r(x_0)}|Du^{\frac{1}{\alpha}}|^p\,dx\\        
        &\le \alpha^p \; \delta \int\limits_{\partial B_r(x_0)}|Du^{\frac{1}{\alpha}}|^{p-1}\,dx\\        
        &=\alpha \; \delta\int\limits_{\partial B_r(x_0)}\left[|Du|u^{-\frac{\gamma}{p}}\right]^{p-1}\,dx\\ 
        &=\alpha \; \delta \left( C\|u\|_{L^\infty(B_1)}^{p-\gamma} \right)^{1-\frac{1}{p}}   \left| \partial B_r(x_0) \right| \\
        &\le C \delta r^{n-1},
    \end{split}
\end{equation*}
where $C>0$ is a constant depending only on $n$, $p$, $\gamma$ and $\|u\|_{L^\infty (B_1)}$. Letting $\varepsilon\to0$, we conclude the proof. 
\end{proof}

\section{Hausdorff measure estimates}\label{Hauss sec}
From \Cref{porosity}, we know that the $n-$dimensional Lebesgue measure of the free boundary is zero. In this section, our goal is to show that, in fact, the $(n-1)-$dimensional Hausdorff measure of the free boundary is finite. While the case $p=2$ is addressed in \cite{P2}, the nonlinear scenario presents substantially greater complexity. Here, we refine the approach developed in \cite{AT13}, which diverges significantly from the methodology in \cite{P2}. The core idea is to derive uniform estimates for minimizers in a flatness regime, leveraging these results alongside \Cref{pointwisegradientestimate}. We combine the estimate in the flatness regime obtained in \Cref{integral estimate in strip} with a measure estimate in the same regime (\Cref{Lebesgue measure estimate lemma}) and a measure estimate in a strip (\Cref{density lemma}) to control stripped neighborhoods of the free boundary.

Recall that the $s-$dimensional Hausdorff measure of a set $E$ is defined by 
$$
\mathcal{H}^s(E):=\lim_{\sigma\to0}\mathcal{H}_\sigma^s(E),
$$
where 
$$
\mathcal{H}_\sigma^s(E):=\inf \left\{ \sum_{j=1}^\infty \mu(s) \left( \frac{\diam \, E_j}{2} \right)^s \right\}.
$$
Here, the infimum is taken over all countable coverings $\{E_j\}$ of $E$ such that $\diam \, E_j \le\sigma$, $\mu(s):=\frac{\pi^{s/2}}{\Gamma \left( \frac{s}{2}+1\right)}$ and $\Gamma(s):=\int_0^\infty e^{-t}t^{s-1}\,dt$, for $s>0$, is the usual Gamma function.

\smallskip

The following result provides a measure estimate in the flatness regime.

\begin{lemma}\label{Lebesgue measure estimate lemma}
    If $u$ is a local minimizer of \eqref{1.1} in $B_1$, and
    $$
    0<\delta \leq r \leq \min\{r_\star,\delta_\star\},
    $$  
    where $\delta_\star>0$ and $r_\star>0$ are defined by \eqref{definition of delta star} and \eqref{definition of r star} respectively, then there exists $C'>0$, depending only on $n$, $p$, $\gamma$ and $\|u\|_{L^\infty (B_1)}$, such that for \( x_0 \in \partial\{ u > 0 \} \cap B_{1/4} \), one has  
\[
\left| \left\{ 0 < u < \delta^{\frac{p}{p-\gamma}} \right\} \cap B_r(x_0) \right| \leq C' \delta r^{n-1}.  
\]  
\end{lemma}

\begin{proof}
Let $x_0\in\partial\{u>0\}\cap B_{1/4}$ and $\{B_j\}$ be a family of balls covering $\partial\{u>0\}\cap B_r(x_0)$, with center at $x_j\in\partial\{u>0\}\cap B_r(x_0)$ and radius $C_0\delta$, for a constant $C_0>1$ to be chosen later. Besicovitch lemma ensures that we can choose such a covering with finite overlaps, \textit{i.e.}, there exists a universal constant $m$ such that
\begin{equation}\label{finite overlap}
        \sum_j \chi_{_{B_j}}\le m,
\end{equation}
where $\chi_{_{B_j}}$ is the characteristic function of the set $B_j$. Note that we can make sure that
$$    \bigcup_jB_j\subset\left[B_{\frac{9}{16}}\cap B_{4r}(x_0)\right].
$$
Moving forward, we divide the proof into three steps.

\medskip
\noindent \textit{Step 1}. We claim that, for each integer $j>0$, there are two balls $B^1_j$, $B^2_j$, with radii $r_j^1$ and $r_j^2$, of order $\delta$, such that
$$
B^1_j \cup B^2_j \subset B_j \qquad \mbox{and} \qquad B^1_j \cap B^2_j = \emptyset.
$$
Moreover,
\begin{equation}\label{choice of phi}  
\omega\ge\delta\left(\frac{3}{4}\right)^{\frac{1}{\alpha}}\ \textrm{in}\ B^1_j \qquad \textrm{and} \qquad \omega\le \delta\left(\frac{2}{3}\right)^{\frac{1}{\alpha}}\ \textrm{in}\ B^2_j,
\end{equation}
where $\omega$, previously defined by \eqref{definition of omega}, is given by
\begin{equation*}
\omega=\left\{
\begin{array}{ccc}
    u^{\frac{1}{\alpha}} & \mbox{in} & \{u < \delta^\alpha\}, \\
     \delta & \mbox{in} & \{u \geq \delta^\alpha\}, 
\end{array}
\right.
\end{equation*}
and $\alpha=\frac{p}{p-\gamma}$, as defined by \eqref{definition of alpha}. Indeed, let $z\in \partial(\frac{1}{4}B_j)$ be such that
$$
u(z)=\sup_{\partial(\frac{1}{4}B_j)}u.
$$
Let $C>0$ be the constant from \Cref{Harnack type inequality}. Using \Cref{non-degeneracy}, 
\begin{equation*}
        u(z)\ge c\left(\frac{C_0\delta}{4}\right)^\alpha\ge C\delta^\alpha,
\end{equation*}
where $C_0$ is chosen large enough, depending only on $n$, $p$, $\gamma$ and $\|u\|_{L^\infty(B_1)}$. Next, for $x\in B_j$, we apply \Cref{Harnack type inequality} in $B_{|x-z|} (x)$, obtaining
$$
u(z) \leq \sup_{B_{|x-z|} (x)} u \leq C\left(|x-z|^\alpha+u(x)\right),
$$
and thus,
$$
u(x)\ge\frac{1}{C}u(z)-|x-z|^\alpha\ge\delta^\alpha-|x-z|^\alpha, \quad x\in B_j.
$$
Taking 
$$
|x-z|\le\frac{\delta}{4^{1/\alpha}},
$$
we obtain
$$
u(x)\ge\frac{3}{4}\delta^\alpha\,\,\,\textrm{ in }\,\,\,B_j^1:=B_{r_j^1}(z),
$$
where $r_j^1:=\delta/4^{1/\alpha}$. It remains to choose $C_0>1$ large enough so that $B_j^1\subset B_j$, which holds if 
$$
|x-z|\leq \frac{\delta}{4^{1/\alpha}} \ \Longrightarrow \ |x-x_j| \le C_0 \delta.
$$
By the triangular inequality, 
$$
|x-x_j| \leq |x-z| + |z-x_j| \leq \frac{\delta}{4^{1/\alpha}}  + \frac{C_0 \delta}{4},
$$
and it is enough to choose
$$
\frac{\delta}{4^{1/\alpha}}  + \frac{C_0 \delta}{4} \leq C_0 \delta \ \Longleftarrow \ C_0 \geq \max \left\{1, \frac{4^{\gamma/p}}{3} \right\}.
$$
Note that this choice is universal.

On the other hand, recalling that $x_j \in \partial\{u>0\}$ and using \Cref{Harnack type inequality} for $B_j^2:= B_{r_j^2}(x_j)$, where $r_j^2:=\left(\frac{2}{3C}\right)^{\frac{1}{\alpha}} \delta<C_0\delta$, one obtains 
$$
\sup_{B_j^2} u \leq C \left( \left(\frac{2}{3C}\right)^{\frac{1}{\alpha}} \delta \right)^\alpha = \frac{2}{3}\delta^\alpha
$$
and the second inequality in \eqref{choice of phi} follows.

\smallskip

\noindent
\textit{Step 2}. Next, consider 
$$
m_1:=\left(\dfrac{3}{4}\right)^{\frac{1}{\alpha}}, \quad m_2:=\left(\dfrac{2}{3}\right)^{\frac{1}{\alpha}}, \quad \kappa :=\dfrac{m_1-m_2}{2}, \quad \mbox{and} \quad m:=\dfrac{m_1+m_2}{2}.
$$
From \eqref{choice of phi}, the following lower bound holds
\begin{equation}\label{beforepoincare}
\left|\omega-(\omega)_j \right| \geq \kappa \delta,
\end{equation}
in at least one of the two balls $B_j^1$ and $B_j^2$, where
$$    
(\omega)_j:=\fint_{B_j}\omega\,dx.
$$
In fact, assuming that $(\omega)_j \leq m \delta$, we have
$$
\left|\omega-(\omega)_j \right| \geq \omega-(\omega)_j \geq m_1\delta-m\delta = \kappa\delta, \quad \mbox{in } \; B^1_j. 
$$
In case $(\omega)_j>m\delta$, we still obtain
$$
\left|\omega-(\omega)_j \right| \geq (\omega)_j-\omega > m\delta -m_2\delta = \kappa\delta, \quad \mbox{in } \; B^2_j. 
$$

\smallskip

\noindent
\textit{Step 3}. Using \eqref{beforepoincare}, Poincar\'e inequality in balls then implies
\begin{equation*}
        C_1 \kappa^p\delta^p \le \frac{1}{|B_j|} \int_{B_j} \left| \omega-(\omega)_j \right|^p\,dx\le C_2\frac{\delta^p}{|B_j|}\int_{B_j}|D\omega|^p\,dx,
\end{equation*}
for constants $C_1, C_2 >0$ universal. Recalling the definition of $\omega$ and \eqref{Gradient of the potential}, we obtain
\begin{equation}\label{important estimate}
        \int\limits_{\left\{0<u<\delta^\alpha\right\} \cap B_j} u^{-\gamma}|Du|^p\,dx\ge C_3|B_j|,
\end{equation}
for a universal constant $C_3>0$.

Furthermore, by \Cref{sharp non-degeneracy}, in $\{u>0\}$ one has
$$
u(x) \ge c_0 \left[d(x)\right]^\alpha,
$$
where $d(x)$ is the distance of $x$ to the free boundary, defined by \eqref{definition of d}, and $c_0>0$ is a universal constant. For $x\in\{0<u<\delta^\alpha\}$ we then have
$$
\left[d(x)\right]^\alpha \leq \frac{u(x)}{c_0} < \frac{\delta^\alpha}{c_0}.
$$
Thus,
$$
d(x) < (1/c_0)^{1/\alpha}\delta =: C_*\delta.
$$
Therefore,
$$
\left\{0<u<\delta^\alpha\right\}\cap B_r(x_0)\subset\mathcal{N}_{C_*\delta}\left( \partial\{u>0\} \right) \cap B_{r}(x_0), 
$$
where
\begin{equation}\label{definition of N}
        \mathcal{N}_\epsilon(E):= \left\{ x\in\R^n : \, \dist(x,E)<\epsilon \right\}.
\end{equation}
Hence, for sufficiently small universal constants $\delta<r$, we have
$$
\left\{0<u<\delta^\alpha\right\}\cap B_r(x_0)\subset\bigcup_j C_* B_j \subset B_{4r}(x_0).
$$

Finally, using \Cref{integral estimate in strip}, \eqref{finite overlap} and \eqref{important estimate}, we estimate
$$
\begin{array}{ccl}
\left| \left\{0<u<\delta^\alpha\right\}\cap B_r(x_0) \right| & \leq & C_* \displaystyle \sum_j |B_j| \\
& \leq & \dfrac{2}{C_3}\displaystyle \sum_j\int\limits_{\left\{0<u<\delta^\alpha\right\}\cap B_j}u^{-\gamma}|Du|^p\,dx\\
& \leq & \dfrac{C_4}{C_3} \delta \displaystyle \sum_j r_j^{n-1} \\
& \leq & \frac{m {C_0}^{n-1}C_4}{C_3} \delta r^{n-1}\\
& = & C^\prime \delta r^{n-1}, 
\end{array}    
$$
for some $C^\prime>0$ universal.
\end{proof}

The following lemma is from \cite[Lemma 4.4]{RT11}. 

\begin{lemma}\label{density lemma}
Let $E \subset \R^n$. If there exists $\tau>0$ such that 
\begin{equation}\label{density esimate}
        \frac{| E \cap B_\delta(x) |}{|B_\delta(x)|}\ge\tau,
\end{equation}
for all $x\in\partial E\cap\Omega$ and some $\delta\in(0,1)$, then
there exists a constant $\tilde{C}>0$, depending only on $n$ and $\tau$, such that
$$
\left| \mathcal{N}_\delta(\partial E)\cap B_r(x) \right| \le \frac{1}{2^n\tau} \left| \mathcal{N}_\delta(\partial E)\cap B_r(x)\cap E \right| + \tilde{C}\delta r^{n-1},
$$
where $r>\delta$, and $\mathcal{N}_\delta$ is defined by \eqref{definition of N}. Furthermore, if \eqref{density esimate} holds for any $\delta\in(0,1)$, then
$$
\left| \partial E\cap\Omega \right| = 0.
$$
\end{lemma}

Before the main result of this section, we combine the previous two lemmas to obtain the following proposition.

\begin{proposition}\label{Lebesgue measure estimate}
Let $\Omega'\Subset \Omega$. If $u$ is a minimizer of \eqref{1.1} and $x_0\in\partial\{u>0\}\cap\Omega'$, then there exists $C>0$ and $r>0$, depending only on $n$, $p$ and $\gamma$, such that, for any $0<\delta \ll r$, there holds
$$    
\left| \mathcal{N}_\delta\left(\partial\{u>0\}\right)\cap B_r(x_0) \right| \leq C\delta r^{n-1},
$$
where $\mathcal{N}_\delta$ is defined by \eqref{definition of N}. 
\end{proposition} 

\begin{proof}
Using \Cref{positivedensity}, we obtain $\tau \in(0,1)$, depending only on $\Omega^\prime$, $n$, $p$ and $\gamma$, such that, for any small ball $B_{\delta}(x_0) \subset \Omega^\prime$, centered at $x_0\in\partial\{u>0\}$, one has
$$
\frac{|\{u>0\}\cap B_{\delta}(x_0)|}{|B_{\delta}(x_0)|}\ge\tau.
$$
Now, from \Cref{density lemma}, with $E=\{u>0\}$, we find $r >\delta$, such that
\begin{equation}\label{Lebesgue measure estimate eq}
\begin{split}
        |\mathcal{N}_\delta(\partial\{u>0\})\cap B_r(x_0)|&\le C_*|\mathcal{N}_\delta(\partial\{u>0\})\cap B_r(x_0)\cap\{u>0\}|\\
        &\quad+\tilde{C}\delta r^{n-1},
\end{split}
\end{equation}
where $C_*>0$ and $\tilde{C}>0$ are universal constants. Observe that if $y\in\mathcal{N}_\delta(\partial\{u>0\})\cap B_r(x_0)\cap\{u>0\}$, then there exists $z\in\partial\{u>0\}$ such that $|y-z|\le\delta$. Employing \Cref{Harnack type inequality}, we get
$$
u(y)\le C\left(|y-z|^\alpha+u(z)\right)\le C\delta^\alpha.
$$
Thus,
\begin{equation}\label{subset}
        \mathcal{N}_\delta(\partial\{u>0\})\cap B_r(x_0)\cap\{u>0\}\subset\{0<u<C\delta^\alpha\}\cap B_r(x_0).
\end{equation}
By \Cref{Lebesgue measure estimate lemma}, we have
\begin{equation}\label{lemma application}
        |\{0<u<C\delta^\alpha\}\cap B_r(x_0)|\le C' {C}^{1/\alpha}\delta r^{n-1}.
\end{equation}
Combining \eqref{Lebesgue measure estimate eq}-\eqref{lemma application}, we get the desired result.    
\end{proof}

We are now ready to prove our main result.

\begin{theorem}\label{Hausdorff}
    Let $\Omega'\Subset \Omega$. If $u$ is a minimizer of \eqref{1.1}, and $x_0\in\partial\{u>0\}\cap\Omega'$, then there exists a constant $C>0$, depending only on $\Omega'$, $n$, $p$ and $\gamma$, such that
    $$
    \mathcal{H}^{n-1}(\partial\{u>0\}\cap B_r(x_0))< Cr^{n-1},
    $$
    for any $r>0$ universally small. 
\end{theorem}

\begin{proof}
Recalling Besicovitch's covering lemma, let $\{B_\sigma(x_i)\}_{i\in I}$ be a finite covering of $\partial\{u>0\}\cap B_r(x_0)$, with $x_i\in\partial\{u>0\}\cap B_r(x_0)$ and $\sigma<r$. Observe that
\begin{equation*}
    \bigcup_i B_\sigma(x_i)\subset\mathcal{N}_\sigma\left(\partial\{u>0\}\right)\cap B_{r+\sigma}(x_0),
\end{equation*}
which together with \Cref{Lebesgue measure estimate} leads to
\begin{eqnarray*}
    \mathcal{H}_\sigma^{n-1}(\partial\{u>0\}\cap B_r(x_0)) & \le & C^\prime\sum_{i\in I}|\partial B_{\sigma}(x_i)|\\
    & = & \frac{C^\prime  n}{\sigma}\sum_{i\in I}|B_{\sigma}(x_i)|\\
    &\le &\frac{C^\prime n}{\sigma} \left| \mathcal{N}_\sigma\left(\partial\{u>0\}\right)\cap B_{r+\sigma}(x_0) \right|  \\
    &\le & C^\prime n (r+\sigma)^{n-1}\\
    &\le& Cr^{n-1},
\end{eqnarray*}
where $C>0$ is a universal constant.
Letting $\sigma\to0$ in the last inequality, we arrive at
$$
\mathcal{H}^{n-1}(\partial\{u>0\}\cap B_r(x_0))\le Cr^{n-1}.
$$
\end{proof}    

\begin{remark}
Since the free boundary has locally finite $\mathcal{H}^{n-1}-$measure, the set $\{u>0\}$ has locally finite perimeter in $\Omega$. Thus, $D(\chi_{\{u>0\}})$ is, in the sense of distributions, a vector-valued Borel measure supported on the free boundary, and its total variation is a Radon measure (see \cite{EG15}). Moreover, up to a negligible set of the null perimeter, the free boundary is a union of, at most, a countable family of $C^1-$hypersurfaces (see \cite{G84}).
\end{remark}

\bigskip

{\small \noindent{\bf Acknowledgments.}} {\footnotesize This publication is based upon work supported by King Abdullah University of Science and Technology (KAUST) under Award No. ORFS-CRG12-2024-6430. DJA is partially supported by the Conselho Nacional de Desenvolvimento Cient\'\i fico e Tecnol\'ogico (CNPq) grants 310020/2022-0 and 420014/2023-3.}

\end{document}